\documentclass[EJP]{ejpecp} 

\usepackage{bm}

\SHORTTITLE{A composite generalization of Ville's martingale theorem using e-processes}

\TITLE{A composite generalization of Ville's martingale theorem using e-processes}

\AUTHORS{%
Johannes Ruf\footnote{Department of Mathematics, London School of Economics.
\EMAIL{j.ruf@lse.ac.uk}}
\and  Martin Larsson\footnote{Department of Mathematical Sciences, Carnegie Mellon University.
\EMAIL{larsson@cmu.edu}}
\and Wouter M. Koolen\footnote{Machine Learning Group, CWI Amsterdam. 
\EMAIL{wmkoolen@cwi.nl}}
\and Aaditya Ramdas\footnote{Departments of Statistics and ML, Carnegie Mellon University. \EMAIL{aramdas@cmu.edu}}
}


\AMSSUBJ{60A10, 60F15} 
\AMSSUBJSECONDARY{60G40} 

\SUBMITTED{May 12, 2021} 
\ACCEPTED{December 13, 2014} 

\VOLUME{0}
\YEAR{2020}
\PAPERNUM{0}
\DOI{10.1214/YY-TN}

\ABSTRACT{We provide a composite version of Ville's theorem that an event has zero measure if and only if there exists a nonnegative martingale which explodes to infinity when that event occurs. This is a classic result connecting measure-theoretic probability to the sequence-by-sequence game-theoretic probability, recently developed by Shafer and Vovk. Our extension of Ville's result involves appropriate composite generalizations of nonnegative martingales and measure-zero events: these are respectively provided by ``e-processes'', and a new inverse capital outer measure. We then develop a novel line-crossing inequality for sums of random variables which are only required to have a  finite first moment, which we use to prove a composite version of the strong law of large numbers (SLLN). This allows us to show that violation of the SLLN is an event of outer measure zero and that our e-process explodes to infinity on every such violating sequence, while this is provably not achievable with a nonnegative (super)martingale.}

\newcommand{\E}{{\mathbb E}}
\newcommand{\N}{{\mathbb N}}
\renewcommand{\P}{{\mathbb P}}
\renewcommand{\d}{{\mathrm d}}
\newcommand{\Q}{{\mathbb Q}}
\newcommand{\R}{{\mathbb R}}

\newcommand{\Gcal}{{\mathcal G}}
\newcommand{\Fcal}{{\mathcal F}}
\newcommand{\Pcal}{{\mathcal P}}

\newcommand{\Tcal}{{\mathcal T}}

\newcommand{\outerNatural}{{\mu}}
\newcommand{\outerLocal}{{\nu}}

\begin{document}



\section{Introduction}
In his PhD thesis, Ville \cite{ville_etude_1939} obtained a martingale characterization of the zero probability events on the filtered probability space generated by an infinite sequence of independent unbiased coin flips. Informally, Ville proved that
\begin{quote} an event $A$ has probability $\P(A) = 0$ if and only if there exists a nonnegative $\P$-martingale $M = (M_t)_{t \in \N_0}$ which starts at one and diverges to infinity on $A$. 
\end{quote}
Here, $\P$ is the probability measure under which the coins are independent and unbiased. 

In game-theoretic terms, Ville proved that for any event $A$ having probability zero, one can construct an explicit betting strategy, which starts at unit wealth and never risks more wealth than it has, such that the wealth increases to infinity if that event $A$ occurs. Via this seminal result, applied specifically to the law of the iterated logarithm (LIL), Ville was instrumental in bringing martingales into the forefront of probability theory. Ville's result is also central to the philosophy of probability theory, because it gives an actionable and concrete interpretation to the relatively abstract concept of a measure-zero event.  Different measure-zero events --- for example, sequences violating the strong law of large numbers (SLLN) and those violating the LIL --- obviously have the same probability, but they result in  different betting strategies.

Although Ville proved his theorem for binary sequences, his result turns out to be far more general, holding true in general filtered spaces. There is also a version for events of positive measure: an event $A$ has $\mathbb{P}$-measure at most $a$ if and only if there exists a $\mathbb{P}$-martingale which starts at one and exceeds any level below $1/a$ if the event $A$ occurs. See Chapter~8.5 in the book by Shafer and Vovk~\cite{shafer2001probability} for a detailed treatment and proof.

Ville's result has implications for sequential hypothesis testing: the nonnegative $\P$-martingale $M$ can be used to construct a sequential test for the null hypothesis that the data sequence was generated according to $\P$. For any $\alpha \in (0,1)$, the test which rejects the null the first time that $M_t \ge 1/\alpha$ has type-I error at most $\alpha$:
\begin{equation}\label{eq:ville}
\P\left( M_t \ge \frac{1}{\alpha} \text{ for some } t \in \N \right) \le \alpha.
\end{equation}
This inequality,  which was proved in the same thesis, is commonly referred to as \emph{Ville's inequality}. Doob, who read Ville's thesis, realized the importance of  martingales and credited this particular inequality to Ville; see Doob~\cite{doob_regularity_1940}.

We briefly mention how this work connects to modern developments in sequential statistics. The idea of constructing nonnegative test statistics that are martingales under the null hypothesis, but increase under relevant alternative hypotheses, has long (implicitly or explicitly) been a cornerstone of sequential statistics; indeed, the sequential probability ratio test by Wald~\cite{wald_sequential_1945} is based on the same principle (sequential likelihood ratios are easily checked to be nonnegative martingales). Following initial work by Robbins~\cite{robbins_statistical_1970}, this idea has been extended to many different nonparametric settings in recent years~\cite{howard_exponential_2018,howard_uniform_2019}; here, 
nonnegative \emph{super}martingales provide a natural and powerful generalization of likelihood ratios.
There are yet other problems for which even nonnegative supermartingales do not suffice, and one encounters a new concept called an e-process~\cite{grunwald_safe_2019,ramdas2020admissible,ramdas2021testing}. We revisit this notion very soon since it plays a fundamental role in the current paper. The aforementioned discoveries underpin a ``game-theoretic'' approach to statistics, which in particular yields hypothesis tests and confidence sets that can be continuously monitored and adaptively stopped; see the recent survey~\cite{ramdas2022game}  and  references therein for an elaboration on such ``safe anytime-valid inference.''

\paragraph{Our contributions.} In this paper we develop a \emph{composite} version of Ville's theorem, where the single probability measure $\P$ is replaced by an entire family $\Pcal$ of probability measures (and the space of coin flips is replaced by a general filtered measurable space).
In this extended setting, two issues must be addressed: 
\begin{enumerate}
    \item First, one needs an appropriate analog of a zero probability event.
    \item Second, one needs an appropriate analog of a martingale. 
\end{enumerate}

A natural replacement for the notion of a $\P$-nullset would be that of a \emph{$\Pcal$-polar set}, which is an event $A$ such that $\P(A) = 0$ for all $\P \in \Pcal$. However, this turns out to be too weak. The appropriate replacement relies on a certain outer measure $\outerLocal \equiv \outerLocal_\Pcal$ introduced in Section~\ref{S_outer_measure}. We call this the \emph{inverse capital measure}, because for any set $A$, the inverse of $\outerLocal(A)$ specifies an asymptotic lower bound on the limiting capital achievable by a gambler (in a particular set of games) starting with an initial capital of one and betting against $\Pcal$, whenever $A$ occurs. 

Next, one may attempt to work with processes that are $\P$-martingales for every $\P \in \Pcal$. However, this requirement turns out to be too strong, and the appropriate replacement is the notion of an \emph{e-process}~\cite{grunwald_safe_2019,ramdas2020admissible,ramdas2021testing}. An e-process for a family $\Pcal$ of probability measures, called a $\Pcal$-e-process for brevity, is a nonnegative process $E = (E_t)_{t \in \N_0}$, adapted to some filtration $(\mathcal F_t)_{t \in \N_0}$, such that 
\begin{equation}\label{def:e-process}
\text{$\E_\P[ E_\tau ] \le 1$ for every $\P \in \Pcal$ and every stopping time $\tau$, }
\end{equation}
where $E_\infty = \limsup_{t \to \infty} E_t$ by convention. (The stopping times above are also defined with respect to $(\mathcal F_t)_{t \in \N_0}$.) 

This is a much weaker property than being a nonnegative $\P$-martingale starting at one for every $\P \in \Pcal$. In particular, the martingale property implies the e-process property thanks to the optional stopping theorem.   An equivalent definition~\cite{ramdas2020admissible} is: an e-process is any nonnegative process that is upper bounded, for every $\P \in \Pcal$, by a $\P$-martingale starting at one (a different one for each $\P$). This style of definition was previously used by Howard et al.~\cite{howard_exponential_2018}, but without using the term e-process, which is a very recent christening of this property. As a consequence of this alternative characterization, e-processes also satisfy Ville's inequality~\eqref{eq:ville}, making their magnitude interpretable and actionable as a measure of evidence.\\

Using the above two notions --- that is, e-processes and the inverse capital measure --- we derive a composite version of Ville's theorem. Specifically, we prove that
\begin{quote}
    an event $A$ has inverse capital measure zero if and only if there exists an e-process that grows to infinity whenever $A$ occurs. 
\end{quote}
We also prove an analogous statement for events with positive inverse capital measure.

E-processes have recently been used to solve a number of problem in sequential statistics, because there exist problems in which there are no nontrivial $\Pcal$-martingales (other than constants), there are no nontrivial $\Pcal$-supermartingales (other than decreasing processes), but there do exist nontrivial $\Pcal$-e-processes which grow to infinity when evaluated on data that is sufficiently incompatible with $\Pcal$. A particularly interesting and fundamental problem where the preceding observations arise, is to test whether a sequence is exchangeable or not \cite{ramdas2021testing}.
A game-theoretic interpretation for e-processes can be found in the aforementioned work, and applications to meta-analysis of medical studies were described in Gr\"unwald et al.~\cite{grunwald_safe_2019}. A thorough treatment of game-theoretic probability can be found in the book by Shafer and Vovk~\cite{shafer2019game}, and a recent case for its use in scientific communication was made in the discussion paper by Shafer~\cite{shafer2019language}.

Our composite version of Ville's theorem generalizes his original result.
Moreover, unpacking the proof of the theorem we see that it constructs relatively explicit gambling strategies for arbitrary inverse capital measure-zero events that gain infinite wealth when the purportedly rare events actually do occur.  
In order to demonstrate a nontrivial example of its applicability, we prove a new \emph{composite strong law of large numbers}. Here the inverse capital measure-zero event $A$ is the set of sequences whose empirical averages do not converge to a limit. In other words, it is possible to \emph{test} whether such a composite SLLN holds, with the evidence (e-process) growing to infinity whenever the SLLN fails. Moreover, we prove that there is no nonnegative martingale or supermartingale that achieves this; the concept of an e-process is necessary for this example.

Proving such a composite strong law is itself nontrivial. Along the way, we develop a new $L^1$-variant of the celebrated $L^2$-type Dubins-Savage inequality. Such inequalities have been extended to settings with finite $p$-th moments for $1<p<2$ by Khan~\cite{khan_$l_p$-version_2009}, and also to higher moments~\cite{howard_exponential_2018}, but the $L^1$-type inequality appears to be new and of independent interest.

To summarize, our main contributions are the following:
\begin{itemize}
\item Constructing a composite generalization of Ville's theorem which informally reads as: for every event of \emph{inverse capital} measure equal to $a \in [0,1]$, we can construct an \emph{e-process} that reaches arbitrarily close to $1/a$ if that event occurs.
\item Deriving a new, $L^1$-type line-crossing inequality, which is an analog of the $L^2$-type Dubins--Savage deviation inequality.
\item Developing a new composite strong law of large numbers, for which there is no martingale or supermartingale that can detect its failure, but there exists an explicit e-process that can.
\end{itemize}

To conclude, it is worth remarking the following. Gambling strategies were called martingales in the 19th century, but Ville switched the term's use to characterize the wealth of the gambler, rather than the gambling strategy itself \cite{mazliak2022splendors}. In some sense, martingales were thrust into the limelight by Ville's result, giving them a central role in modern ``measure-theoretic'' probability theory, and also demystifying measure zero events to some extent. In a similar sense, our generalization gives an independent justification for the claim that the e-process is an equally fundamental concept for dealing with composite nulls (as martingales are for point nulls): the e-process arises organically when one tries to generalize Ville's seminal result to the composite setting.

\section{The inverse capital measure} \label{S_outer_measure}
Consider a filtered measurable space $(\Omega,\Fcal,(\Fcal_t)_{t \in \N})$ where 
$\Fcal = \sigma(\bigcup_{t \in \N} \Fcal_t)$, and let $\Tcal$ denote the set of all stopping times with respect to this filtration. Fix an arbitrary set $\Pcal$ of probability measures on $\Fcal$ and define a $[0,1]$-valued set function $\outerLocal$ by 
\begin{equation}\label{eq:inv-cap-outer-m}
\outerLocal(A) = \inf_{\tau \in \Tcal: ~ A \subseteq \{\tau < \infty\}}  ~ \sup_{\P \in \Pcal} ~ \P(\tau < \infty),
\end{equation}
for any subset $A \subseteq \Omega$. We call this the \emph{inverse capital} measure in anticipation of its later role in the composite generalization of Ville's theorem. Note that $\outerLocal$ depends both on the filtration $(\Fcal_t)_{t \in \N}$ and the set $\Pcal$.

To interpret $\outerLocal(A)$ one should think of $A$ as a \emph{failure event}. An example that we study in detail below is the event that a certain adapted process $(\overline X_t)_{t \in \N}$ fails to converge as $t$ tends to $\infty$. One then looks for stopping times $\tau$ that detect failure in finite time, meaning that $A \subseteq \{ \tau < \infty\}$. Among all such stopping times one searches for those that are \emph{efficient}, uniformly across $\Pcal$, in the sense that $\sup_{\P \in \Pcal} \P(\tau < \infty)$ is minimal. This can be thought of as minimizing the probability of false alarms.

There is also an interpretation in terms of type-I error control for sequential tests. Indeed, the definition of $\outerLocal(A)$ says that for any $\alpha > \outerLocal(A)$ there exists a level-$\alpha$ sequential test $\tau$ for the null hypothesis $\Pcal$, with the additional property that the null is guaranteed to be rejected if $A$ occurs. Here a stopping time $\tau$ is identified with the test that rejects the null when $\tau$ occurs, and the level-$\alpha$ property means that the null, if true, will be rejected with probability at most $\alpha$; that is, $\sup_{\P \in \Pcal} \P(\tau < \infty) \le \alpha$. A further interpretation in terms of e-processes arises from Theorem~\ref{T_e_process} below.

There is another natural outer measure, which we call the \emph{maximum likelihood (outer) measure}, namely 
\begin{equation} \label{eq_nu_star}
\outerNatural(A) = \sup_{\P \in \Pcal} \P(A).
\end{equation}
Note that $\outerNatural(A) \leq \outerLocal(A)$ for every $A \in \Fcal$. The inequality can be strict, a fact which is demonstrated and discussed further in Subsection~\ref{S_nu_star} below.

\subsection{The inverse capital measure is an outer measure}

The following lemma gathers some properties of $\outerLocal$.

\begin{lemma} \label{L_properties_of_mu_star}
The set function $\outerLocal$ satisfies the following properties.
\begin{enumerate}
\item\label{L_properties_of_mu_star_1} Monotonicity: if $A \subseteq B$ then $\outerLocal(A) \le \outerLocal(B)$.
\item\label{L_properties_of_mu_star_2} Countable subadditivity: for any sequence of sets $A_n \subseteq \Omega$, $n \in \N$, one has
\[
\outerLocal\Big( \bigcup_{n \in \N} A_n \Big) \le \sum_{n \in \N} \outerLocal(A_n).
\]
\item\label{L_properties_of_mu_star_3} Singletons: if $\Pcal = \{ \P \}$ consists of a single element, then $\outerLocal(A) = \P(A)$ for all $A \in \Fcal$.
\end{enumerate}
In particular, since clearly $\outerLocal(\emptyset) = 0$, the set function $\outerLocal$ is an outer measure on $\Omega$.
\end{lemma}

\begin{proof} Property~ \ref{L_properties_of_mu_star_1} is immediate from the definition of $\outerLocal$. Let us next prove the remaining two assertions.

Property~\ref{L_properties_of_mu_star_2}: Consider sets $A_n \subseteq \Omega$, $n \in \N$, and fix any $\varepsilon > 0$. For each $n$, let $\tau_n$ be a stopping time such that $A_n \subseteq \{\tau_n < \infty\}$ and $\sup_{\P \in \Pcal} \P(\tau_n < \infty) \le \outerLocal(A_n) + \varepsilon 2^{-n}$. Then $\tau = \inf_{n \in \N} \tau_n$ is a stopping time which satisfies $\bigcup_{n \in \N} A_n \subseteq \{\tau < \infty\}$ and, for each $\P \in \Pcal$,
\[
\P( \tau < \infty ) = \P\Big( \bigcup_{n \in \N} \{\tau_n < \infty\} \Big) \le \sum_{n \in \N} \P(\tau_n < \infty) \le \varepsilon + \sum_{n \in \N} \outerLocal(A_n).
\]
Taking supremum over $\P \in \Pcal$ and using that $\varepsilon > 0$ was arbitrary, the result follows.

Property~\ref{L_properties_of_mu_star_3}: Let $A \in \Fcal$, as assumed. The definition of $\outerLocal$ gives $\P(A) \le \outerLocal(A)$. For the reverse inequality, fix any $\varepsilon > 0$. Since $\Fcal$ is generated by $\bigcup_{t \in \N}\Fcal_t$, the standard Carath\'eodory construction~\cite[Section~3]{billingsley2008probability} that extends $\mathbb{P}$ to $\mathcal{F}$ shows that we can cover $A$ by a union $B = \bigcup_{t \in \N} A_t$ of disjoint events $A_t \in \Fcal_t$ such that $B \supseteq A$ and $\P(B) \le \P(A) + \varepsilon$. 
This ensures that the random variable $\tau = \sum_{t \in \N} t \bm1_{A_t} + \infty \bm1_{B^c}$ (meaning that $\tau$ equals infinity on $B^c$) is a stopping time with $\{\tau < \infty\} = B$. Thus by definition, $\outerLocal(A) \le \P(\tau < \infty) \le \P(A) + \varepsilon$. Since $\epsilon>0$ was arbitrary, we conclude that $\outerLocal(A) \le \P(A)$.
\end{proof}

\subsection{The inverse capital measure versus the maximum likelihood measure} \label{S_nu_star}

We begin by noting that the maximum likelihood measure $\outerNatural$ defined in \eqref{eq_nu_star} also satisfies properties \ref{L_properties_of_mu_star_1}--\ref{L_properties_of_mu_star_3} of Lemma~\ref{L_properties_of_mu_star}. As mentioned before,
it follows from the definitions that $\outerNatural(A) \le \outerLocal(A)$ for every $A \in \Fcal$. 
In the singleton case $\Pcal = \{\P\}$ the inequality is actually an equality since $\outerNatural(A) = \outerLocal(A) = \P(A)$ by Lemma~\ref{L_properties_of_mu_star}(\ref{L_properties_of_mu_star_3}). This implies, in particular, that the maximum likelihood measure can be expressed as
\begin{equation*}
    \outerNatural(A) =  \sup_{\P \in \Pcal} ~ \inf_{\tau \in \Tcal: ~ A \subseteq \{\tau < \infty\}}   \P(\tau < \infty),
\end{equation*}
which provides a direct comparison to~\eqref{eq:inv-cap-outer-m}. Moreover, for general $\Pcal$, we have $\outerLocal(A) = \outerNatural(A)$ for any ``finite-time'' event $A$ of the form $A = \{\tau < \infty\}$. 
However, the inequality can be strict in general. In fact, we show in Example~\ref{ex_Cauchy} that it is possible to have $\outerNatural(A) = 0$ but $\outerLocal(A) = 1$. In particular, the collection of $\outerLocal$-nullsets can be strictly smaller than the collection of $\Pcal$-polar sets, which we recall are sets $A$ such that $\P(A)=0$ for all $\P\in\Pcal$. A key message of Theorem~\ref{T_e_process}, and indeed of this paper, is that in the context of sequential testing and composite versions of Ville's results, $\outerLocal$ turns out to be the more relevant object, and not $\outerNatural$.

\begin{example} \label{ex_Simple}
Consider the sample space $\Omega = \{0,1\}^\N$, the space of all binary sequences. Let $X = (X_t)_{t \in \N}$ denote the canonical process and set $\Fcal_t = \sigma(X_s, s \le t)$, $t \in \N$. Let $\P_n$ assign probability $1/2$ to the sequence $0^\N$ (all zeros) and to the sequence $0^n \cdot 1 \cdot 0^\N$ (a single one occurs after $n$ zeros). We consider $\Pcal = \{ \P_n \}_{n \in \N_0}$. Let $A = \{0^\N\}$. Clearly $\P_n(A) = 1/2$ for all $n \in \N_0$, and hence $\outerNatural(A) = 1/2$. 

  However, we have $\outerLocal(A) = 1$. To see why, observe that any stopping time $\tau$ that stops on $A$ must stop after some finite number of zeros, say $m$. But then for $n \ge m$, we have $\P_n(\tau < \infty) = 1$. Since this holds for any $\tau$ such that $A \subseteq \{\tau < \infty\}$, we conclude $\outerLocal(A)=1$. 
  
  Curiously, on the complementary event $A^c$ (there is a one somewhere), $\outerLocal(A^c) = 1/2$. This claim is witnessed by the stopping time that stops when it encounters a one.
\end{example}

\begin{remark}
In the singleton case of $\Pcal=\{\mathbb P\}$, which is what Ville considered, the outer measure $\outerLocal$ of any measurable set $A$ equals its probability $\P(A)$, which does not depend on the filtration. In the general composite case, however, the underlying filtration matters in the definition of $\outerLocal$ because the filtration determines the set of stopping times. In particular, the set of $\outerLocal$ nullsets depends on the filtration. This can be seen by considering the extreme case where the filtration $(\Fcal_t)_{t \in \N_0}$ is replaced by the constant filtration $(\Gcal_t)_{t \in \N_0}$ where $\Gcal_t = \Fcal$ for all $t$. In this case $\outerLocal$ becomes equal to the maximum likelihood measure $\outerNatural$ in \eqref{eq_nu_star}. In general, a larger filtration leads to smaller values of $\outerLocal$ (but, of course, it cannot decrease below $\outerNatural$) and to a larger family of increasing e-processes. In this way, more information can lead to more powerful tests.

One may ask for conditions under which $\outerLocal$ and $\outerNatural$ are equal on $\Fcal$. An abstract condition of this kind can be extracted directly from the definition of $\outerLocal$. Indeed, the definition yields
\[
\outerNatural(A) \le \outerLocal(A) \le \outerNatural(A) + \inf_{\tau \in \Tcal\colon A \subseteq \{\tau < \infty\}} \outerNatural(\{\tau < \infty\} \setminus A).
\]
Thus $\outerLocal$ and $\outerNatural$ coincide on $\Fcal$ provided that any $A \in \Fcal$ admits a \emph{uniform outer approximation} by sets of the form $\{\tau < \infty\}$, $\tau \in \Tcal$, in the sense that there is a nondecreasing sequence of stopping times $\tau_n$ such that $A \subset \{\tau_n < \infty\}$ for all $n$ and $\sup_{\P \in \Pcal} \P(\{\tau_n < \infty\} \setminus A) \downarrow 0$ as $n \to \infty$.
\end{remark}

\begin{remark}
For certain choices of $\Pcal$, the set function $\outerNatural$ is \emph{submodular}, meaning that
\[
\outerNatural(A \cup B) + \outerNatural(A \cap B) \le \outerNatural(A) + \outerNatural(B)
\]
for all $A,B \in \Fcal$. In this case, submodularity is inherited by $\outerLocal$. Indeed, suppose $A \subset \{\tau < \infty\}$ and $B \subset \{\sigma < \infty\}$ for some $\tau,\sigma \in \Tcal$. Then, because $\{\tau < \infty\} \cup \{\sigma < \infty\} =  \{\tau \wedge \sigma < \infty\}$ and $\{\tau < \infty\} \cap \{\sigma < \infty\} = \{\tau \vee \sigma < \infty\}$, the definition of $\outerLocal$ and submodularity of $\outerNatural$ yield
\begin{align*}
\outerLocal(A \cup B) + \outerLocal(A \cap B) &\le \outerNatural(\tau \wedge \sigma < \infty) + \outerNatural(\tau \vee \sigma < \infty) \\
&\le \outerNatural(\tau < \infty) + \outerNatural(\sigma < \infty).
\end{align*}
Taking the infimum of the right-hand side over all such stopping time $\tau, \sigma$ yields submodularity of $\outerLocal$, namely
\[
\outerLocal(A \cup B) + \outerLocal(A \cap B) \le \outerLocal(A) + \outerLocal(B).
\]
However, $\outerNatural$ is not always submodular. For example, consider the four-point space $\Omega = \{a,b,c,d\}$ along with $\Pcal = \{\P_1,\P_2\}$ where
\[
\P_1 = \left(\frac13,\frac13,\frac13,0\right), \quad \P_2 = \left(0,\frac23,0,\frac13\right).
\]
In this case we have for $A=\{a,b\}$ and $B=\{b,c\}$ that $\outerNatural(A \cup B) + \outerNatural(A \cap B)=\frac53$ but $\outerNatural(A) + \outerNatural(B)=\frac43$. This violates submodularity. As a consequence, since $\outerLocal$ may coincide with $\outerNatural$ (for instance if the filtration is constant), $\outerLocal$ may fail to be submodular as well. Submodularity of $\outerLocal$, when it holds, implies that the Choquet integral $\int f d\outerLocal$ with respect to $\outerLocal$ of any bounded $\Fcal$-measurable real-valued function $f$ on $\Omega$ admits the representation
\[
\int f d\outerLocal = \max_{\P \in \Pcal^*} \E_\P[f],
\]
where $\Pcal^*$ is the set of all finitely additive probability measures $\P$ on $\Fcal$ such that $\P(A) \le \outerLocal(A)$ for all $A \in \Fcal$. The Choquet integral is then both subadditive and comonotonically additive. See \cite{MR3859905} for more details, in particular Theorem~4.88 and Corollary~4.89. It is an interesting open question what can be said about the relation between $\Pcal^*$, $\Pcal$, and the filtration $(\Fcal_t)_{t \in \N_0}$. We would like to thank an anonymous referee for bringing the question of submodularity to our attention.
\end{remark}

\subsection{The law of large numbers for truncated Cauchy distributions}\label{sec:cauchy}

We now present a more sophisticated example that is of relevance later in the paper. A key added feature compared to Example~\ref{ex_Simple} is that we now consider i.i.d.\ sequences.

\begin{example} \label{ex_Cauchy}
Consider the sample space $\Omega = \R^\N$, the space of all infinite real-valued sequences. Let $X = (X_t)_{t \in \N}$ denote the canonical process and set $\Fcal_t = \sigma(X_s, s \le t)$, $t \in \N$. We will consider a collection $\Pcal$ of laws under which $X$ is i.i.d.\ with Cauchy distributions that are truncated at higher and higher levels. This will lead to a situation where $\outerNatural(A) = 0$, $\outerLocal(A) = 1$, and $\outerLocal(A^c) = 1$ hold simultaneously for the same event $A = A_\textnormal{div}$ given by
\begin{equation}\label{eq:a-div}
A_\textnormal{div} = \left\{ \lim_{t \to \infty} \frac{1}{t} \sum_{s=1}^t X_s \textnormal{ does not exist} \right\}.
\end{equation}
Specifically, let $\Q$ denote the law under which $X$ is i.i.d.\ Cauchy (centered at the origin). For each $a \in (0,\infty)$ we define the truncation of $X$ at magnitude $a$ by
\[
X^a_t = X_t \bm1_{\{|X_t| < a\}} + a \bm1_{\{X_t \ge a\}} - a \bm1_{\{X_t \le -a\}}, \qquad t \in \N.
\]
Then, under $\Q$, the process $X^a = (X^a_t)_{t \in \N}$ is i.i.d.\ with a truncated Cauchy distribution. More precisely, the marginal distribution coincides with the Cauchy on $(-a,a)$ and collects the remaining mass symmetrically at $\pm a$. We let $\P_a$ denote the law of $(X^a_t)_{t \in \N}$ under $\Q$, and consider the family
\[
\Pcal = \{\P_a \colon a \in (0, \infty) \}.
\]
\end{example}

\begin{proposition}
\label{prop:cauchy}
For the class $\Pcal$ of truncated Cauchy distributions defined above, we have
\begin{equation} \label{eq_ex_Cauchy}
\outerNatural(A) =0, \text{ but }
\outerLocal(A_\textnormal{div}) = 1 = \outerLocal(A_\textnormal{div}^c).
\end{equation}
\end{proposition}
\begin{proof}
For each $a \in (0, \infty)$ we have $\E_{\P_a}[ |X_1| ] = \E_{\Q}[ |X^a_1| ] \le a$ and $\E_{\P_a}[ X_1 ] = \E_{\Q}[ X^a_1 ] = 0$. Thus the strong law of large numbers yields
\[
\P_a(A_\textnormal{div}) = 0,
\]
for any $a>0$. Thus, $\sup_{\P \in \Pcal} \P(A) = \outerNatural(A)= 0$, as claimed. 

In order to reason about $\outerLocal$, consider any stopping time $\tau$ such that $A_\text{div} \subseteq \{\tau < \infty\}$. We write $\tau = \tau(X)$ to emphasize that $\tau$ is a function of the trajectory $X = (X_t)_{t \in \N}$. We have
\begin{equation} \label{eq_ex_Cauchy_2}
\begin{aligned}
\P_a( \tau(X) < \infty ) &\ge \P_a( \tau(X) < \infty \text{ and } |X_t| < a \text{ for all } t \le \tau(X) ) \\
&= \Q( \tau(X^a) < \infty \text{ and } |X^a_t| < a \text{ for all } t \le \tau(X^a) ) \\
&= \Q( \tau(X) < \infty \text{ and } |X_t| < a \text{ for all } t \le \tau(X) ),
\end{aligned}
\end{equation}
where in the last equality we used that $X^a_t = X_t$ whenever either value belongs to $(-a,a)$, and that $\tau(X)$ only depends on the values of $X_t$ for $t \le \tau(X)$. Next, since $X$ is i.i.d.\ Cauchy under $\Q$, and hence the sample averages almost surely fail to converge, we have $\Q(\tau(X) < \infty) \ge \Q(A_\textnormal{div}) = 1$. Thus 
\begin{equation} \label{eq_ex_Cauchy_3}
\begin{aligned}
\Q( & \tau(X) < \infty \text{ and } |X_t| < a \text{ for all } t \le \tau(X) ) \\
&= \Q( |X_t| < a \text{ for all } t \le \tau(X) ) \\
&\ge \Q( |X_t| < a \text{ for all } t \le T ) - \Q( \tau(X) > T)
\end{aligned}
\end{equation}
for any $T \in \N$. Fix any $\varepsilon > 0$ and choose $T$ large enough that $\Q( \tau(X) > T) \le \varepsilon$, which is possible because $\Q(\tau(X) < \infty) = 1$. Combining \eqref{eq_ex_Cauchy_2} and \eqref{eq_ex_Cauchy_3} we then deduce that
\[
\P_a( \tau(X) < \infty ) \ge \Q( |X_t| < a \text{ for all } t \le T ) - \varepsilon.
\]
Sending $a$ to $\infty$ we find that $\sup_{a \in (0,\infty)} \P_a( \tau(X) < \infty ) \ge 1 - \varepsilon$. Since $\tau$ was an arbitrary stopping time with $A_\textnormal{div} \subseteq \{\tau < \infty\}$ it follows that $\outerLocal(A_\textnormal{div}) \ge 1 - \varepsilon$. Since this holds for every $\varepsilon > 0$, we obtain the first equality in \eqref{eq_ex_Cauchy}.

Despite the fact that \eqref{eq_ex_Cauchy} holds, we also have $\outerLocal(A_\textnormal{conv}) = 1$ where 
\[
A_\textnormal{conv} = A_\textnormal{div}^c = \left\{ \lim_{t \to \infty} \frac{1}{t} \sum_{s=1}^t X_s \textnormal{ exists} \right\}.
\]
Indeed, $\outerLocal(A_\textnormal{conv}) \ge \outerNatural(A_\textnormal{conv}) = \sup_{\P \in \Pcal} \P(A_\textnormal{conv}) = 1$ by the strong law of large numbers.
\end{proof}

Note that \eqref{eq_ex_Cauchy} is not a contradiction, since $\outerLocal$ is an outer measure, and not a probability measure.

Much intuition from the above example will be necessary when we prove a composite version of the strong law of large numbers in Section~\ref{sec:slln}, where we impose a uniform integrability condition on the class, in order for $A_\textnormal{div}$ to be a $\outerLocal$-nullset (as one may expect), as opposed to a full outer-measure set in the above example.

\section{A composite generalization of Ville's theorem}
\subsection{Statement and proof of the composite version}

We now have the necessary technical instruments in place to present our first main result, which is a composite version of Ville's theorem in terms of e-processes (recall their defining property in \eqref{def:e-process}).

\begin{theorem}\label{T_composite_Ville}
Let $(\Omega, \mathcal{F}, (\mathcal{F}_t)_{t \in \N_0})$ be a filtered measurable space with $\mathcal{F} = \sigma(\bigcup_{t \in \N_0} \mathcal{F}_t)$, let $\Pcal$ be a family of probability measures on $\Fcal$, let $\outerLocal$ be the inverse capital outer measure from~\eqref{eq:inv-cap-outer-m}, and let $A \subseteq \Omega$ be arbitrary. Then, 
\[
\text{$\outerLocal(A) = 0$ if and only if there exists a $\Pcal$-e-process $(E_t)_{t \in \N_0}$ such that $\lim_{t \to \infty} E_t = \infty$ on $A$.}
\]
\end{theorem}

In the context of the above theorem, any $\Pcal$-e-process that explodes to infinity on $A$ will be called a ``witness e-process for $A$''. The proof below explicitly constructs such a witness e-process, and we give even more concrete witness e-processes in special cases later.

\begin{proof}
If $\outerLocal(A) = 0$, then for each $n \in \N$ there exists a stopping time $\tau_n$ such that $A \subseteq \{\tau_n < \infty\}$ and $\sup_{\P \in \Pcal} \P(\tau_n < \infty) \le 2^{-n}$. Define 
\begin{equation} \label{T_composite_Ville_eq1}
E_t = \sum_{n \in \N} \bm1_{\{\tau_n \leq t\}}, \qquad t \in \N_0.
\end{equation}
This process is nondecreasing and satisfies $E_\infty = \lim_{t \to \infty} E_t = \sum_{n \in \N} \bm1_{\{\tau_n < \infty\}}$. In particular, $E_\infty = \infty$ on $A$. Moreover, $(E_t)_{t \in \N_0}$ is an e-process because for any stopping time $\tau$ and $\P \in \Pcal$,
\[
\E_\P[ E_\tau ] \le \E_\P[ E_\infty ] = \sum_{n \in \N} \P(\tau_n < \infty) \le 1.
\]
Conversely, suppose there is an e-process $(E_t)_{t \in \N_0}$ such that $\lim_{t \to \infty} E_t = \infty$ on $A$. Pick any $a \in (0,\infty)$ and let $\tau = \inf\{t \ge 0 \colon E_t \ge a\}$. Then $A \subseteq \{\tau < \infty\}$, and the definition of an e-process yields
\[
\P( \tau < \infty ) \le \frac{1}{a} \E_\P [ \bm1_{\{\tau < \infty\}} E_\tau ] \le \frac{1}{a}
\]
for all $\P \in \Pcal$. Thus $\outerLocal(A) \le 1/a$ and hence, since $a \in (0,\infty)$ was arbitrary, $\outerLocal(A) = 0$.
\end{proof}

Several remarks are in order.

\begin{remark}[The singleton case]\label{rem:singleton}
In the singleton case $\Pcal = \{\P\}$, this result does not quite recover the original theorem of Ville, because our e-process does not reduce to being a $\P$-martingale. This is easily remedied by modifying \eqref{T_composite_Ville_eq1} to $E_t = \sum_{n \in \N} \P(\tau_n < \infty \mid \Fcal_t)$, which is a $\P$-martingale. To be precise, we must choose versions of the conditional probabilities $\P(\tau_n < \infty \mid \Fcal_t)$ that are identically equal to one on $\{\tau_n \le t\}$ (not just almost surely equal to one, which is automatic), in order to obtain $\lim_{t \to \infty} E_t = \infty$ everywhere on $A$ (not just almost everywhere). The original proof by Ville, or the  version in Chapter~8.5 of \cite{shafer2001probability}, carefully accounts for this subtlety.
\end{remark}

\begin{remark}[$\outerNatural(A)=0$ does not suffice]
We clarify here that a version of the theorem with $\outerLocal(A)=0$ being replaced by $\outerNatural(A)=0$ is incorrect. Indeed, Section~\ref{sec:cauchy} demonstrated an explicit event (for truncated Cauchy distributions) with $\outerNatural(A)=0$ but $\outerLocal(A)=1$. If $\outerNatural(A)=0$ had sufficed for producing a witness $\Pcal$-e-process (or supermartingale, or martingale) for $A$, then the converse direction of Theorem~\ref{T_composite_Ville} could be invoked to then further conclude that $\outerLocal(A)=0$, which yields a contradiction. 
\end{remark}

\begin{remark}[Nonnegative $\Pcal$-supermartingales do not suffice]\label{rem:supermartingales}
In the general non-singleton case, one cannot expect an e-process to be a martingale under every $\P \in \Pcal$. (We demonstrate an example in Subsection~\ref{sec:necessity} for which no $\Pcal$-supermartingale can increase to infinity, but an e-process can.) However, in the special case that all elements of $\Pcal$ are mutually locally absolutely continuous
and $\Pcal$ satisfies a certain geometric condition known as ``fork-convexity'', then it is possible to construct the e-process to be a nonnegative $\Pcal$-supermartingale. This requires technical machinery involving composite Snell envelopes, that we do not develop further in this paper; see  \cite[Section~3]{ramdas2021testing}.
\end{remark}

\begin{remark}[Countable unions of events]\label{rem:unions}
If one considers a countable number of events $\{A_n\}_{n \in \N}$ such that $\outerLocal(A_n)=0$ for each of them, Lemma~\ref{L_properties_of_mu_star} implies that their union also satisfies $\outerLocal(\bigcup_n A_n)=0$. This fact is witnessed by simply taking any mixture of the underlying e-processes. To clarify, if $E^n$ is the e-process from the above theorem which witnesses the fact that $\outerLocal(A_n)=0$, then for any positive sequence $\{w_n\}_{n \in \N}$ that sums to one, we see that $\sum_{n\in\N} w_n E^n$ is an e-process that witnesses $\outerLocal(\bigcup_n A_n)=0$. Indeed, $\sum_{n\in\N} w_n E^n$ is infinite if and only if at least one of its constituent e-processes is also infinite. (This remark is a direct parallel of Lemma~3.2 in \cite{shafer2001probability}, except the remark is stated more generally in terms of e-processes.)
\end{remark}

\begin{remark}[An unbounded $\sup$ suffices]\label{rem:limsup-suffices}
We claim that if there exists a $\Pcal$-e-process $E$ such that $\sup_{t\ge0} E_t = \infty$ on an event $A$, then there also exists a possibly different $\Pcal$-e-process $E'$ such that $\lim_{t\to\infty} E'_t = \infty$. The proof is simple. Let $\tau_n=\inf\{t: E_t \geq 2^n\}$ be the first time that $E$ crosses level $2^n$. Since $E$ crosses every finite level on $A$, we know that $A \subseteq \{\tau_n < \infty\}$.  Now define $E'_t=\sum_{n=1}^\infty \bm1_{\{\tau_n \leq t\}}$, exactly as in~\eqref{T_composite_Ville_eq1}, and note that on $A$, we have $\lim_{t\to\infty} E'_t=\infty$. Since $E'$ is adapted, nondecreasing, and 
\[
\mathbb{E}_{\mathbb{P}}\left[\lim_{t \uparrow \infty} E'_t\right] = 
\mathbb{E}_{\mathbb{P}}\left[ \sum_{n=1}^\infty \bm1_{\{\tau_n < \infty\}}\right]
=\sum_{n=1}^\infty \mathbb{P}(\tau_n < \infty)
\leq \sum_{n=1}^\infty \frac{1}{2^n} = 1
\]
for all $\mathbb{P} \in \Pcal$, $E'$ is indeed an e-process.
In summary, if one can exhibit an e-process $E$ that is unbounded on $A$, then one can conclude that $\outerLocal(A)=0$. (This remark is a direct parallel of Lemma~3.1 in \cite{shafer2001probability}, except the remark is stated more generally in terms of e-processes.)
\end{remark}

\subsection{SubGaussian supermartingales and e-processes as witnesses}\label{eg:subGaussian}

We now provide two simple examples where a nonnegative $\Pcal$-supermartingale does bear witness to our main theorem, and a third example where a $\Pcal$-e-process emerges. 

We first recall a standard definition: a zero-mean random variable $X$ is called $\sigma$-subGaussian, for some known $\sigma > 0$,  if $\mathbb{E}[e^{\lambda X}] \leq e^{\lambda^2\sigma^2/2}$ for every $\lambda \in \mathbb{R}$. A centered Gaussian with variance at most $\sigma^2$ is $\sigma$-subGaussian, but \cite{hoeffding_probability_1963} proved that centered bounded random variables are also subGaussian (for an appropriate $\sigma$). Thus, the class of subGaussian random variables is very richly nonparametric, containing discrete and continuous distributions, both bounded and unbounded. 

\paragraph{Case 1: zero mean.}
Now, define $\Pcal^0$ to consist of all distributions over real-valued sequences $(X_t)_{t \in \N}$ such that for every $t\in\N$,  $X_t$ is mean-zero and $\sigma$-subGaussian when conditioned on $X_1,\dots,X_{t-1}$. Recalling~\eqref{eq:a-div}, define the event
\begin{equation}\label{eq:a-div-zero}
A^0_\textnormal{div} = \left\{ \lim_{t \to \infty} \frac{1}{t} \sum_{s=1}^t X_s \textnormal{ does not converge to zero}\right\},
\end{equation}
and note that $\mathbb{P}(A^0_\textnormal{div})=0$ for every $\mathbb{P}\in\Pcal^0$. We claim that $\outerLocal(A^0_\textnormal{div})=0$. It is possible to argue this directly, but we will instead do so by exhibiting a nonnegative $\Pcal^0$-supermartingale whose $\limsup$ is unbounded on $A^0_\textnormal{div}$. We treat the case $\sigma=1$ without loss of generality. First note that it is an immediate consequence of conditional subGaussianity that for any $\lambda \in \mathbb{R}$, the process 
\begin{equation}\label{eq:subG-lambda-supermg}
L_t(\lambda) = \exp\left(\lambda \sum_{i=1}^t X_i - \lambda^2 t/2\right)
\end{equation}
is a nonnegative $\Pcal^0$-supermartingale with respect to the natural filtration generated by the data. By mixing this supermartingale over all $\lambda \in \mathbb{R}$ using a standard Gaussian distribution over $\lambda$, \cite{robbins_boundary_1970} derived the ``normal mixture'' $\Pcal^0$-supermartingale 
\begin{equation*}
M_t = \frac{\exp\left(\frac{t^2 \overline X_t^2}{ 2(t+1)}\right)}{\sqrt{t+1}},
\end{equation*}
where $\overline X_t =  \sum_{i=1}^t X_i/t$ is the empirical mean. 

Now, we claim that  $(M_t)_{t \in \N_0}$ is unbounded on $A^0_\textnormal{div}$. Indeed, if $\overline X_t$ does not converge to zero, then, along a subsequence, $\overline X_t^2$ remains above a strictly positive constant $c$. Along this subsequence, we have $M_t \ge \exp(  t^2 c^2 / (2(t+1)))/\sqrt{t+1}$, which converges to infinity. By Remark~\ref{rem:limsup-suffices}, it suffices to treat $(M_t)_{t \in \N_0}$ as the witness  $\Pcal^0$-supermartingale (and thus $\Pcal^0$-e-process) to conclude that $\outerLocal(A^0_\textnormal{div})=0$.

\paragraph{Case 2: nonpositive mean.}
Next, consider the event
\begin{equation*}
A^{>0} = \left\{ \limsup_{t \to \infty} \frac{1}{t} \sum_{s=1}^t X_s > 0 \right\},
\end{equation*}
and the set of distributions 
\begin{equation*}
\Pcal^{\leq0}=\left\{\mathbb{P}:  
a_t \leq 0 \text{ and } X_t - a_t \text{ is conditionally 1-subGaussian}, \text{ for all } t\in\N \right\},
\end{equation*}
where $a_t = \mathbb{E}_\mathbb{P}[X_t \mid X_1,\dots,X_{t-1}]$ exists, is finite, and can itself be predictable. 

In this case, the process $L_t(\lambda)$ from~\eqref{eq:subG-lambda-supermg} is still a $\Pcal^{\leq 0}$-supermartingale, but only for $\lambda \geq 0$. Then, mixing over $\lambda$ using only the positive half of a standard Gaussian, we get the ``one-sided normal mixture'' $\Pcal^{\leq 0}$-supermartingale, for which there is no closed form expression \cite{robbins_boundary_1970,howard_uniform_2019}. Luckily, a suitable closed-form lower bound is given by
\begin{equation*}
N_t = 2 \frac{\exp\left(\frac{2t^2 \overline X_t^2 }{t+1}\right) }{\sqrt{t+1}} \Phi\left(\frac{2t \overline X_t}{\sqrt{t+1}}\right),
\end{equation*}
where $\Phi$ is the standard Gaussian cumulative distribution function. 
Since $(N_t)_{t \in \N_0}$ is not itself a supermartingale, but is upper bounded by one, it is a $\Pcal^{\leq 0}$-e-process.
Now, a similar argument as before proves that $(N_t)_{t \in \N_0}$ is unbounded on $A^{>0}$, that is, $\limsup_{t \to \infty} N_t = \infty$ on any sequence of observations such that $\limsup_{t \to \infty} \overline X_t > 0$.

\paragraph{Case 3: nonpositive running mean.} Finally, consider the class of distributions
\[
\Pcal^\dagger =\left\{\mathbb{P}: 
\sum_{s \leq t} a_s \leq 0 \text{ and } X_t - a_t \text{ is conditionally 1-subGaussian, for all } t\in\N \right\},
\]
where, as before, $a_t = \mathbb{E}_\mathbb{P}[X_t \mid X_1,\dots,X_{t-1}]$ exists, is finite, and can itself be predictable.

We have $\Pcal^\dagger \supsetneq \Pcal^{\leq 0}$ since the definition of $\Pcal^\dagger$ allows for a sequence of conditional means like $-1, -1, 2, 0, -2, 2, \dots$, whose individual elements could be positive as long as the cumulative sums are nonpositive. In this case, the process $L_t(\lambda)$ from~\eqref{eq:subG-lambda-supermg} is no longer a $\Pcal^\dagger$-supermartingale, but it is a $\Pcal^\dagger$-e-process for all $\lambda \geq 0$. (To see this, simply note that for each $\P \in \Pcal^\dagger$, the process $L_t(\lambda)$ is upper bounded by a different nonnegative $\P$-supermartingale, namely $\exp(\lambda \sum_{s=1}^t (X_s-a_s) - \lambda^2 t/2)$.) Thus $N_t$ is also a $\Pcal^\dagger$-e-process, and since we already know that it is unbounded on $A^{>0}$ we can conclude that $\outerLocal(A^{>0})=0$ even for $\Pcal^\dagger$.

We end this subsection by noting that there was nothing particularly special about the above subGaussian example. Howard et al.~\cite{howard_exponential_2018,howard_uniform_2019} describe large classes of composite supermartingales and e-processes under different types of assumptions on the underlying random variable (like subexponential, sub-gamma, sub-Poisson, etc.), for which analogous sets of events and distributions could have been defined, and similar arguments could have been made.

\subsection{E-process characterization for events with positive measure}

One might wonder if there is a version of Theorem~\ref{T_composite_Ville} that is valid for $\outerLocal(A) > 0$. The following result and subsequent discussion and example show what is possible.

\begin{theorem}\label{T_e_process}
Let $(\Omega, \mathcal{F}, (\mathcal{F}_t)_{t \in \N_0})$ be a filtered measurable space with $\mathcal{F} = \sigma(\bigcup_{t \in \N_0} \mathcal{F}_t)$, let $\Pcal$ be a family of probability measures on $\Fcal$, let $\outerLocal$ be the inverse capital outer measure from~\eqref{eq:inv-cap-outer-m}, and let $A \subseteq \Omega$ be arbitrary. Then, the following claims are true:
\begin{enumerate}
\item\label{T_e_process_1} For each $\varepsilon > 0$ there exists an e-process $(E_t)_{t \in \N_0}$ such that $\lim_{t \to \infty} E_t = 1 / (\outerLocal(A) + \varepsilon)$ on $A$.
\item\label{T_e_process_2} For any $c \in [1,\infty]$, if there exists an e-process $(E_t)_{t \in \N_0}$ such that $\sup_{t \in \N_0} E_t \ge c$ on $A$, then $\outerLocal(A) \le 1/c$.
\end{enumerate}
\end{theorem}

\begin{proof}
\ref{T_e_process_1}: Let $\tau_\varepsilon$ be a stopping time with $A \subseteq \{\tau_\varepsilon < \infty\}$ and $\sup_{\P \in \Pcal} \P(\tau_\varepsilon < \infty) \le \outerLocal(A) + \varepsilon$, and define the one-jump process
\[
E_t = \frac{1}{\outerLocal(A) + \varepsilon} \bm1_{\{\tau_\varepsilon \leq t\}}, \qquad t \in \N_0.
\]
Then $E_\infty = \bm1_{\{\tau_\varepsilon < \infty\}}/(\outerLocal(A) + \varepsilon)$, which equals $1 / (\outerLocal(A) + \varepsilon)$ on $A$. Moreover, $(E_t)_{t \in \N_0}$ is an e-process because for any stopping time $\tau$ and $\P \in \Pcal$,
\[
\E_\P[ E_\tau ] \le \E_\P[ E_\infty ] = \frac{1}{\outerLocal(A) + \varepsilon} \P(\tau_\varepsilon < \infty) \le 1.
\]

\ref{T_e_process_2}: This is proved exactly as the converse implication of Theorem~\ref{T_composite_Ville}, except that the constant $a$ is chosen from $(0,c)$. One obtains $\outerLocal(A) \le 1/a$ and hence, since $a \in (0,c)$ was arbitrary, $\outerLocal(A) \le 1/c$.
\end{proof}

\begin{remark}
Let $\Pcal$ be the family of i.i.d.\ truncated Cauchy distributions constructed in Example~\ref{ex_Cauchy}. As shown there, $\outerLocal(A_\textnormal{div}) = 1$ which, in view of Theorem~\ref{T_e_process}\ref{T_e_process_2}, makes it impossible to refute convergence with uniform power for this family $\Pcal$. On the other hand, we also saw that $\outerLocal(A_\textnormal{conv}) = 1$ so refuting \emph{divergence} is also impossible for $\Pcal$.
\end{remark}

\subsection{A counterexample to a natural conjecture}

The e-processes constructed in part \ref{T_e_process_1} of Theorem~\ref{T_e_process} depend on $\varepsilon$, and it is natural to ask whether one can find a single e-process $(E_t)_{t \in \N_0}$ that satisfies
\begin{equation}\label{eq_lim_Et_1_muA}
\text{$\sup_{t \in \N_0} E_t \ge \frac{1}{\outerLocal(A)}$ on $A$}
\end{equation}
when $\outerLocal(A) \in (0,1)$. In the singleton case $\Pcal = \{\P\}$ this can be done, at least if $A$ is measurable and $\P$-nullsets are ignored, using the Doob martingale $E_t = \P(A \mid \Fcal_t) / \P(A)$. Indeed, L\'evy's zero-one law states that $\lim_{t \to \infty} E_t = \bm1_A / \P(A)$ almost surely, and this equals $1 / \outerLocal(A)$ on $A$ due to Lemma~\ref{L_properties_of_mu_star}\ref{L_properties_of_mu_star_3}. However, if $\Pcal$ is not a singleton, it is in general impossible to find a single e-process that satisfies \eqref{eq_lim_Et_1_muA}. The following example shows that this phenomenon can occur even if $\Pcal$ contains just two elements.

\begin{example}
Consider the space $\Omega = \{0,1\}^\N$ of binary sequences (coin flips), let $X = (X_t)_{t \in \N}$ denote the canonical process, and set $\Fcal_t = \sigma(X_s, s \le t)$, $t \in \N$. We define two probability measures $\P_1$ and $\P_2$ as follows. The first coin flip is unbiased,
\[
\P_i(X_1 = 1) = \frac12, \qquad i=1,2.
\]
Conditionally on $X_1$, the remaining coin flips $X_2,X_3,\ldots$ are independent under both $\P_1$ and $\P_2$ with success probabilities
\begin{align*}
\P_1( X_t = 1 \mid X_1 = 1) &= \P_2( X_t = 1 \mid X_1 = 0) = \frac12; \\
\P_1( X_t = 1 \mid X_1 = 0) &= \P_2( X_t = 1 \mid X_1 = 1) = 2^{-t}
\end{align*}
for $t \ge 2$. Thus under $\P_1$, if $X_1=1$ the remaining coin flips are i.i.d.\ Bernoulli-$(1/2)$, while if $X_1=0$ they are independent Bernoulli with exponentially decaying (hence summable) success probabilities. Under $\P_2$ the situation is symmetric.

Consider now the event
\[
A = \{ X_t = 1 \text{ for infinitely many $t \in \N$} \}.
\]
The Borel--Cantelli lemma yields $\P_i(A) = 1/2$ for $i=1,2$. Moreover, if $\tau_n$ is the first time that $n$ ones have been observed, then $A \subset \{\tau_n < \infty\}$ and $\P_i(\tau_n < \infty)$ converges to $1/2$ as $n$ tends to infinity, for $i=1,2$. It follows that $\outerLocal(A) = 1/2$.

However, we claim that no e-process for $\Pcal = \{\P_1,\P_2\}$ can satisfy \eqref{eq_lim_Et_1_muA}. To show this, consider any e-process $(E_t)_{t \in \N_0}$ and suppose that $\sup_{t \in \N_0} E_t \ge c$ on $A$ for some constant $c \in [1,\infty)$. We must show that $c < 2$. To do so, fix $T \in \N$ and $\varepsilon \in (0,c)$ such that $\{E_T \ge \varepsilon\} \ne \emptyset$. Define $B = \{E_T \ge \varepsilon\} \cap \{X_1 = 0\}$ and consider the case $B \ne \emptyset$ (the complementary case $\{E_T \ge \varepsilon\} \cap \{X_1 = 1\} \ne \emptyset$ is argued symmetrically). Denote $\delta = \P_1(B) \in (0, 1/2]$. (Indeed $\delta > 0$ since $B \neq \emptyset$ by assumption.) Next, define the stopping time
\[
\tau = \inf\{t \colon E_t \ge c - \delta \varepsilon\} \wedge T_B,
\]
where $T_B = T \bm1_B + \infty \bm1_{B^c}$ (meaning $T_B$ takes the value infinity on the complement of $B$). Note that both $A$ and $B$ are subsets of $\{\tau < \infty\}$. Moreover, the Borel--Cantelli lemma yields $A = \{X_1 = 0\}$ up to a $\P_1$-nullset, and hence $\P_1(A \cap B) = 0$. We conclude that
\[
\E_{\P_1} [ E_\tau \bm1_{\{\tau < \infty\}}] \ge \E_{\P_1} [ E_\tau \bm1_A ] + \E_{\P_1} [ E_\tau \bm1_B ]
\ge \frac12 ( c - \delta \varepsilon ) + \delta \min\{c-\delta \varepsilon, \varepsilon\} > \frac{c}{2}.
\]
On the other hand, the left-hand side is bounded by one since $(E_t)_{t \in \N_0}$ is an e-process. This shows that $c < 2$, as required.
\end{example}

\subsection{Gambling interpretations of the inverse capital measure}

Nonnegative martingales have a straightforward game-theoretic interpretation: they are the wealth of a gambler playing a fair game. Nonnegative supermartingales are similar: they are the wealth of a gambler playing a game whose odds are stacked against them (or playing a fair game, and throwing away some money in each round). $\Pcal$-e-processes admit a slightly more sophisticated game-theoretic interpretation: at the very beginning, the gambler agrees to play (in parallel) a separate fair game against each $P \in \Pcal$ (or one with odds stacked against them), and their wealth at any time is given by their minimum wealth across all these games. For more on this interpretation, see Section~5.4 in Ramdas et al.~\cite{ramdas2021testing}. Thus it is natural to ask: does $\outerLocal$ also admit a game-theoretic interpretation?

In this subsection, we rewrite Theorem~\ref{T_e_process} to bring to the fore its pricing interpretation of $\outerLocal$. The following corollary is a composite analog of Proposition~8.13 in Shafer and Vovk~\cite{shafer2001probability}. It shows that  $1/\outerLocal(A)$ is the maximum level that can be reached by an e-process started at 1, whenever the event $A$ occurs.

\begin{corollary}\label{C_eprocess_price}
Let $\Pcal$ be a family of probability measures on $\Omega$ and let $\outerLocal$ be the associated inverse capital measure~\eqref{eq:inv-cap-outer-m}. Then
\begin{equation} \label{eq:startingcap}
\outerLocal(A) =
    \inf \left\{c \in (0,1] : \liminf_{t \to \infty} c E_t \ge \bm1_A
    \right\}
=
\inf \left\{c \in (0,1] : \sup_{t\in \N_0} c E_t \ge \bm1_A
    \right\},
\end{equation}
where $(E_t)_{t \in \N_0}$ ranges over the set of all e-processes for $\Pcal$.
\end{corollary}

\begin{proof}
  Let $V_{\lim\inf}$ and $V_{\sup}$ denote the value of the first and second infimum respectively in~\eqref{eq:startingcap}. As $\lim\inf_{t \to \infty} c E_t \le \sup_{t \in \N_0} c E_t$, the $\inf$ defining $V_{\sup}$ is over a larger set, and hence $V_{\sup} \le V_{\lim\inf}$. To then prove the claimed equalities \eqref{eq:startingcap}, it suffices to show the sandwich $V_{\lim\inf} \leq \outerLocal(A) \le V_{\sup}$.  In order to prove the first inequality, Theorem~\ref{T_e_process}(\ref{T_e_process_1}) yields an e-process showing that $c=\outerLocal(A)+\epsilon$ is feasible for $V_{\lim\inf}$ for any $\epsilon > 0$, which implies the first inequality by letting $\epsilon \to 0$. Moving on to the second inequality,  Theorem~\ref{T_e_process}(\ref{T_e_process_2}) shows that for $V_{\sup}$ every feasible pair of $c$ and e-process must have $c \ge \outerLocal(A)$, proving the second inequality. This completes the proof.
\end{proof}

The following rewrite brings out yet another aspect related to betting. For a process $(M_t)_{t \in \N_0}$, thought of as payoffs, we can say it has fair price $\sup_{\rho \in \cal T}~ \E_\P[M_\rho]$ under $\P$. The following result shows that the inverse capital measure of an event $A$ is also the infimum of all initial capitals that suffice to super-replicate the payoff $\bm{1}_A$. Due to this interpretation we could also have plausibly called the inverse capital measure the ``initial capital measure''.

    \begin{lemma}
    The following dual representation for $\outerLocal$ holds:
      \begin{equation}
        \outerLocal(A)
        ~\stackrel{=}~
        \inf_{\tau \in \cal T : A \subseteq \{\tau < \infty\}}~
        \sup_{\P \in \Pcal}~
        \P(\tau < \infty)
        ~=~
        \inf_{\substack{
            \text{$(M_t)_t \ge 0$ s.t.} \\
            \sup_{t \in \N_0} M_t \ge \bm{1}_A
          }
          }~
        \sup_{\P \in \Pcal}~
        \sup_{\rho \in \cal T}~
        \E_\P[M_\rho].
    \end{equation}
  \end{lemma}
  \begin{proof}
  The first expression is simply the definition of $\outerLocal(A)$, and is included for convenience of comparison to the second expression.
    We prove both directions in turn.
    \begin{itemize}
    \item
      For $\ge$, fix any feasible $\tau$ on the left. Then $M_t = \bm{1}_{\tau < \infty}$ is feasible on the right, by monotonicity in time it is best to take $\rho=\infty$, upon which $\E_\P [M_\rho] = \P(\tau<\infty)$.

    \item
      For $\le$, fix any feasible $(M_t)_t$ on the right. For any $\epsilon \in (0,1)$, consider the stopping time $\tau = \inf \{t : M_t \ge 1-\epsilon\}$, so that $A \subseteq \{\tau < \infty\}$ and hence $\tau$ is feasible on the left. Moreover, $\P(\tau <\infty)(1-\epsilon) \le \E_\P[M_\tau] \le \sup_\rho \E_\P[M_\rho]$. Taking $\epsilon \to 0$ proves the claim.
    \end{itemize}
This completes the proof.
  \end{proof}

\section{Line-crossing inequalities and a composite SLLN}
\label{sec:slln}

Before we introduce our composite strong law of large numbers (SLLN), we present an important technical tool that may be of independent interest.

\subsection{An $L^1$-type line-crossing inequality}

Recall the following famous inequality due to Dubins and Savage~\cite{dubins_tchebycheff-like_1965}: for any square integrable martingale $(M_t)_{t \in \N}$ with $M_0 =0$, having predictable quadratic variation $V_t = \sum_{s=1}^t \E[(M_s-M_{s-1})^2|\Fcal_{s-1}]$, we have for any constants $\gamma,\epsilon>0$,
\[
\P\left( \sup_{t \in \N} \frac{M_t}{\gamma+V_t} \geq \epsilon \right) \leq \frac1{\gamma \epsilon^2+1}.
\]

The following theorem can be understood as an $L^1$-type Dubins--Savage inequality. Namely, it gives a probability bound on the likelihood that a random walk starting at zero with i.i.d.\ centered increments (that are only required to be in $L^1$), ever crosses a line with positive slope and intercept. To the best of our knowledge, such inequalities have only been proven in the presence of higher moments for the random walk increments; see, for example, Khan~\cite{khan_$l_p$-version_2009} for $L^p$ variants for $1 < p \leq 2$, and Howard et al.~\cite{howard_exponential_2018} for other variants under two, three, and infinite moments.
It might be of interest to improve the constants in the statement, although this is not needed for our purposes. 

\begin{theorem} \label{T_line_crossing}
Suppose that $(X_t)_{t \in \N}$ is i.i.d.\ with $\E[|X_1|] < \infty$ and $\E[X_1] = 0$ and define the tail function
\begin{equation} \label{T_line_crossing_eq0}
r(K) = \E\left[ |X_1 | \bm1_{\{|X_1 | > K\}} \right], \qquad K \geq 0.
\end{equation}
Then for any positive numbers $\varepsilon$, $\gamma$, and any $K \geq 1$ one has
\[
\P\left( \sup_{t \in \N} \frac{1}{\gamma + t} \left| \sum_{s=1}^t X_s \right| > \varepsilon + r(K) \right)
\le
\frac{8 K^2}{ \gamma \varepsilon^2} + \left( \frac{16}{\varepsilon^2} + 2 \right) r(K).
\]
\end{theorem}

\begin{proof}
We use the triangle inequality to bound
\begin{equation} \label{T_line_crossing_eq1}
\frac{1}{\gamma + t} \left| \sum_{s=1}^t X_s \right|
\le
\frac{1}{\gamma + t} \left| \sum_{s=1}^t Y_s \right|
+
\frac{1}{t} \left| \sum_{s=1}^t Z_s \right|
+
\frac{1}{t} \left| \sum_{s=1}^t R_s \right|,
\end{equation}
where
\begin{align*}
Y_s &= X_s \bm1_{\{|X_s| \le K\}} - \E[ X_s \bm1_{\{|X_s| \le K\}} ], \\
Z_s &= X_s \bm1_{\{K < |X_s| \le s - 1\}} - \E[ X_s \bm1_{\{K < |X_s| \le s - 1\}} ], \\
R_s &= X_s \bm1_{\{|X_s| > K \vee (s - 1)\}} - \E[ X_s \bm1_{\{|X_s| > K \vee (s - 1)\}} ].
\end{align*}
We handle the three terms on the right-hand side of \eqref{T_line_crossing_eq1} one by one.

\textit{Step 1:} Since the $Y_s$ are centered i.i.d.\ with variance bounded by $K^2$, the classical Dubins--Savage line-crossing inequality \cite{dubins_tchebycheff-like_1965} implies that
\begin{equation}
\P\left( \exists t \in \N \text{ with }  \sum_{s=1}^t Y_s > \beta + \alpha K^2  t  \right) \le \frac{1}{1 + \alpha \beta} \le \frac{1}{\alpha \beta}
\end{equation}
for any positive numbers $\alpha,\beta$. We apply this with $\alpha = \varepsilon / K^2$ and $\beta = \varepsilon \gamma$ to get the following tail bound for the first term on the right-hand side of \eqref{T_line_crossing_eq1}:
\begin{equation} \label{T_line_crossing_eq2}
\P\left( \sup_{t \in \N} \frac{1}{\gamma + t} \left| \sum_{s=1}^t Y_s \right| > \varepsilon \right)
\le 2\, \P\left( \exists t \in \N \text{ with }  \sum_{s=1}^t Y_s > \varepsilon \gamma + \varepsilon  t  \right)
\le \frac{2 K^2}{\gamma \varepsilon^2}.
\end{equation}

\textit{Step 2:} Summation by parts yields the identity
\begin{equation}
\frac1t \sum_{s=1}^t Z_s = \sum_{s=1}^t \frac{Z_s}{s} - \frac1t \sum_{j=2}^t \sum_{s=1}^{j-1} \frac{Z_s}{s}.
\end{equation}
Writing $S = \sup_{t \in \N} | \sum_{s=1}^t {Z_s}/{s} |$ we deduce $ \left| \sum_{s=1}^t Z_s \right|/t \le S + \sum_{j=1}^t S/t = 2S$ for all $t \in \N$. 
Noting that $\sum_{s = 1}^t Z_s/s$ is a martingale with independent increments, we may use Kolmogorov's maximal inequality and monotone convergence to get
\begin{equation}
\P\left( \sup_{t \in \N} \frac1t \left| \sum_{s=1}^t Z_s \right| > \varepsilon \right) \le \P\left( \sup_{t \in \N} \left| \sum_{s=1}^t \frac{Z_s}{s} \right| > \frac{\varepsilon}{2} \right) \le \frac{4}{\varepsilon^2} \E\left[ \sum_{s=1}^\infty \frac{Z_s^2}{s^2} \right].
\end{equation}
Since $\E[Z_s^2] \le \E[ X_1^2 \bm1_{\{K < |X_1| \le s - 1\}} ]$, the expectation on the right-hand side can be bounded as follows:
\[
\E\left[ \sum_{s=1}^\infty \frac{Z_s^2}{s^2} \right] \le \E\left[ X_1^2 \bm1_{\{|X_1| > K\}} \sum_{s=1}^\infty \bm1_{\{|X_1| \le s-1\}} \frac{1}{s^2} \right]
\le \E[ |X_1| \bm1_{\{|X_1| > K\}}],
\]
using in the last step the inequality $\sum_{s=1}^\infty \bm1_{\{x \le s-1\}} {1}/{s^2} \le \int_x^\infty {1}/{s^2}\d s = {1}/{x}$, $x \in (0,\infty)$. This establishes the tail bound
\begin{equation} \label{T_line_crossing_eq3}
\P\left( \sup_{t \in \N} \frac1t \left| \sum_{i=1}^t Z_s \right| > \varepsilon \right)
\le \frac{4}{\varepsilon^2} r(K),
\end{equation}
where $r(K)$ is the tail function \eqref{T_line_crossing_eq0}.

\textit{Step 3:} Write $R_s = X_s \bm1_{\{|X_s| > K \vee (s - 1)\}} - c_s$ with $c_s = \E[ X_s \bm1_{\{|X_s| > K \vee (s - 1)\}} ]$. On the event where $|X_s| \bm1_{\{|X_s| > K\}} \le s-1$ for all $s \in \N$, one has $R_s = -c_s$ for all $s \in \N$, and hence
\[
\frac{1}{t} \left| \sum_{s=1}^t R_s \right| = \frac{1}{t} \left| \sum_{s=1}^t c_s \right| \le \frac{1}{t} \sum_{s=1}^t \E[ |X_s| \bm1_{\{|X_s| > K \}} ] = \E[ |X_1| \bm1_{\{|X_1| > K \}} ] = r(K)
\]
for all $t \in \N$. Therefore,
\[
\P\left( \text{$|X_s| \bm1_{\{|X_s| > K\}} \le s-1$ for all $s \in \N$} \right)
\le
\P\left( \sup_{t \in \N} \frac{1}{t} \left| \sum_{s=1}^t R_s \right| \le r(K) \right).
\]
We now take complements, apply the union bound, and use that $\sum_{n \in \N} \P(\xi > n-1) \le \P(\xi > 0) + \E[\xi]$ for any nonnegative random variable $\xi$ to get
\begin{equation} \label{T_line_crossing_eq4}
\begin{aligned}
\P\left( \sup_{t \in \N} \frac{1}{t} \left| \sum_{s=1}^t R_s \right| > r(K) \right)
&\le
\sum_{s \in \N} \P\left( |X_1| \bm1_{\{|X_1| > K\}} > s-1 \right) \\
&\le \P(|X_1| > K) + \E[ |X_1| \bm1_{\{|X_1| > K\}}] \\
&\le 2 \E[ |X_1| \bm1_{\{|X_1| > K\}}] = 2 r(K).
\end{aligned}
\end{equation}
Here the third inequality follows because $K \ge 1$.

\textit{Step 4:} Combining \eqref{T_line_crossing_eq1}, \eqref{T_line_crossing_eq2}, \eqref{T_line_crossing_eq3}, and \eqref{T_line_crossing_eq4} yields
\[
\P\left( \sup_{t \in \N} \frac{1}{\gamma + t} \left| \sum_{s=1}^t X_s \right| > 2 \varepsilon + r(K) \right)
\le
\frac{2 K^2}{\gamma \varepsilon^2} + \left( \frac{4}{\varepsilon^2} + 2 \right) r(K).
\]
Replacing $\varepsilon$ by $\varepsilon/2$ completes the proof.
\end{proof}

By choosing $K = \gamma^{1/3}$ one obtains the following corollary. In particular, this makes it clear that the left-hand side converges to $0$ as $\gamma$ tends to $\infty$, and that this occurs at a rate that only depends on the tail function $r(K)$. Recall that $r(K)$ decreases to $0$ as $K$ tends to $\infty$.

\begin{corollary} \label{C_line_crossing_1}
In the setting of Theorem~\ref{T_line_crossing} one has
\begin{align} \label{eq:230318.1}
\P\left( \sup_{t \in \N} \frac{1}{\gamma + t} \left| \sum_{s=1}^t X_s \right| > 2\varepsilon \right)
\le
\frac{8}{\varepsilon^2} \gamma^{-1/3} + \left( \frac{16}{\varepsilon^2} + 2 \right) r(\gamma^{1/3})
\end{align}
for any positive $\varepsilon> 0$ and any large enough $\gamma \geq 0$ such that $r(\gamma^{1/3}) \le \varepsilon$. Since 
\[\frac{1}{k}  \Big| \sum_{s=1}^k X_s  \Big| \le 2 \sup_{t \in \N} \frac{1}{k + t}  \Big| \sum_{s=1}^t X_s \Big|,
\] we further conclude that
\begin{align} \label{eq:230318.2}
\P\left( \frac{1}{k} \left| \sum_{s=1}^k X_s \right| > \varepsilon \right)
\le
\frac{128}{\varepsilon^2} k^{-1/3} + \left( \frac{256}{\varepsilon^2} + 2 \right) r(k^{1/3})
\end{align}
for any positive $\varepsilon$ and any $k \in \N$ such that $r(k^{1/3}) \le \varepsilon$.
\end{corollary}

\subsection{A composite strong law of large numbers}\label{sec:comp-SLLN}

The line-crossing inequality of Theorem~\ref{T_line_crossing} implies a composite strong law of large numbers. Specifically, if $\outerLocal$ is induced by a family $\Pcal$ of i.i.d.\ laws that satisfy a certain uniform integrability condition, then the event that the time averages of $(X_t)_{t \in \N}$ diverge,
\[
A_\textnormal{div} = \left\{ \text{the limit $\lim_{t \to \infty} \frac1t \sum_{i=1}^t X_i$ does not exist} \right\},
\]
will be a $\outerLocal$-nullset. 

\begin{theorem} \label{T_LLN}
Let $\Pcal$ be a family of probability measures $\P$ under which $(X_t)_{t \in \N}$ is i.i.d.\ with $\E_\P[ |X_1| ] < \infty$. Assume also that $\Pcal$ satisfies the centered uniform integrability condition
\begin{equation} \label{T_LLN_eq1}
\lim_{K \to \infty} \sup_{\P \in \Pcal} \E_\P\left[ |X_1 - \E_\P[X_1] | \bm1_{\{|X_1 - \E_\P[X_1]| > K\}} \right] = 0.
\end{equation}
Then $\outerLocal(A_\textnormal{div}) = 0$.
\end{theorem}

Let us contrast this composite SLLN result with classical SLLNs. The latter assume some condition (e.g.\ i.i.d.\ with finite variance) and prove that $A_\textnormal{div}$ has measure zero. That is, they show that for  any set of distributions $\Pcal$ satisfying the assumed condition, the maximum likelihood measure is zero, that is $\outerNatural(A_\textnormal{div}) = 0$. This in turn implies that $\P(A_\textnormal{div}) = 0$ for every $\P \in \Pcal$, providing a strong argument for why $A_\textnormal{div}$ will not happen. Our composite SLLN instead strengthens this to $\outerLocal(A_\textnormal{div}) = 0$. This implies that there is a $\Pcal$-e-process $(E_t)_{t \in \N_0}$ that explodes when $A_\textnormal{div}$ occurs. Equivalently, for every confidence level $\delta \in (0,1)$, there is a stopping time $\tau$ that detects (by stopping at some finite time) that $A_\textnormal{div}$ is occurring, with a type-I error bounded by $\delta$ for all $\P \in \Pcal$.
In words, $\outerNatural(A_\textnormal{div}) = 0$ means that convergence occurs, while $\outerLocal(A_\textnormal{div}) = 0$ means that we can in addition detect divergence in finite time.

An example of $\Pcal$ that satisfies condition~\eqref{T_LLN_eq1} is a location family created by translating a distribution with finite mean.
To elaborate, let $\phi(x)$ be a density such that $\int |x| \phi(x) dx < \infty$. Abusing notation to use $\phi(x)$ to also denote the distribution with that density, define $\Pcal^\phi= \{\phi(x-a)^\infty \}_{a \in \mathbb R}$ as distributions yielding i.i.d.\ draws from some shifted version of $\phi$. Then $\Pcal^\phi$ satisfies~\eqref{T_LLN_eq1}. In fact, more can be said. For any given set of i.i.d.\ product distributions $\Pcal$, let $\Pcal^0$ be the family of distributions obtained by centering each of the distributions in $\Pcal$, so that every $\P\in\Pcal^0$ yields zero-mean i.i.d.\ draws. Then $\Pcal$ satisfies~\eqref{T_LLN_eq1} if and only if $\Pcal^0$ satisfies the more standard uniform integrability condition, which we know (from the de la Vall\'ee-Poussin theorem) holds if and only if there exists a nonnegative increasing function $G$ such that $\lim_{x \to \infty} G(x)/x = \infty$ and $\sup_{\P \in \Pcal^0} \mathbb E[G(|X_1|)] < \infty$.

An example of $\Pcal$ that fails condition~\eqref{T_LLN_eq1} is the set of all truncated zero-centered Cauchy distributions encountered in~Example~\ref{ex_Cauchy}. This example, and especially Proposition~\ref{prop:cauchy}, reminds us that if uniform integrability conditions are dropped entirely, then $\outerLocal(A_\textnormal{div})$ will not be zero, but could even be equal one. Thus, condition~\eqref{T_LLN_eq1} is not only sufficient for concluding $\outerLocal(A_\textnormal{div})=0$, but also that some condition like~\eqref{T_LLN_eq1} is necessary to restrict $\Pcal$. However, we leave the exact characterization of the necessary condition for future work.

\begin{proof}[Proof of Theorem~\ref{T_LLN}]
We use the notation $\overline X_t =  \sum_{s=1}^t X_s/t$. For each $n \in \N$, define the event
\begin{equation}
A_n = \left\{ \limsup_{t \to \infty} \overline X_t - \liminf_{t \to \infty} \overline X_t > \frac2n \right\}.
\end{equation}
Then $A_\text{div} \subseteq \bigcup_{n \in \N} A_n$, so we wish to bound $\outerLocal(A_n)$. To this end, observe that for any $k,n \in \N$ one has $A_n \subseteq \{\tau_{k,n} < \infty\}$, where
\begin{equation} \label{T_LLN_eq2}
\tau_{k,n} = \inf\left\{ t \ge 0 \colon | \overline X_{k+t} - \overline X_k | > \frac1n \right\}.
\end{equation}
(Indeed, if $\tau_{k,n} = \infty$, $\overline X_{k+t}$ stays within $1/n$ of $\overline X_k$ for all $t \ge 0$. But then so do both $\limsup_{t \to \infty} \overline X_t$ and $\liminf_{t \to \infty} \overline X_t$. This shows that $\{\tau_{k,n} = \infty\} \subseteq A_n^c$.) We conclude that
\[
\outerLocal(A_n) \le \sup_{\P \in \Pcal} \P( \tau_{k,n} < \infty).
\]
It is shown below that the right-hand side tends to $0$ as $k$ tends to $\infty$, for each $n \in \N$. Thus $\outerLocal(A_n) = 0$ for all $n \in \N$, and hence $\outerLocal(A_\text{div}) \le \sum_{n \in \N} \outerLocal(A_n) = 0$ as required.

It remains to show, for any fixed $n \in \N$, that $\sup_{\P \in \Pcal} \P( \tau_{k,n} < \infty)$ tends to $0$ as $k$ tends to $\infty$. Observe that for any $k \in \N$ and $c \in \R$,
\begin{equation}\label{eq_T_LNN_5}
| \overline X_{k+t} - \overline X_k | \le \frac{1}{k+t} \left| \sum_{i=1}^t (X_{k+i} - c) \right| + 2 | \overline X_k - c |.
\end{equation}
Thus for any $\P \in \Pcal$ and $k \in \N$ we may take $c = \E_\P[ X_1 ]$ to obtain
\[
\P( \tau_{k,n} < \infty) \le \P \left( \sup_{t \in \N} \frac{1}{k+t} \left| \sum_{i=1}^t (X_{k+i} - \E_\P[X_1]) \right| \ge \frac{1}{2n} \right)
+ \P \left( | \overline X_k - \E_\P[X_1] | \ge \frac{1}{4n} \right).
\]
The line-crossing inequality \eqref{eq:230318.1} and the inequality \eqref{eq:230318.2}  in Corollary~\ref{C_line_crossing_1} 
now show that for $k$ large enough the right-hand side is bounded by
\[
C \left( k^{-1/3} + \E_\P\left[ |X_1 - \E_\P[X_1] | \bm1_{\{|X_1 - \E_\P[X_1]| > k^{1/3}\}} \right] \right)
\]
for a constant $C$ which, for instance, can be taken to be $C =  4352 n^2 + 4$.  Thanks to the assumption \eqref{T_LLN_eq1}, it follows that $\sup_{\P \in \Pcal} \P( \tau_{k,n} < \infty)$ tends to $ 0$ as $k$ tends to $\infty$. This completes the proof of the theorem.
\end{proof}

\begin{remark}
With $\Pcal$ as in Theorem~\ref{T_LLN}, this result together with the composite Ville's theorem, Theorem~\ref{T_composite_Ville}, shows that there exists an e-process for $\Pcal$ which converges to infinity on $A_\textnormal{div}$. An inspection of the proofs shows that this e-process can be chosen as
\[
E_t = \sum_{n \in \N} \bm1_{\{\tau_{k_n,n} \le t\}}, \qquad t \in \N_0,
\]
where $\tau_{k_n,n}$ is given by \eqref{T_LLN_eq2} with $k=k_n$, and $k_n \gg n^{6}2^{3n}$ is chosen large enough that
\[
(4352 n^2 + 4) \left( k_n^{-1/3} + \sup_{\P \in \Pcal} \E_\P\left[ |X_1 - \E_\P[X_1] | \bm1_{\{|X_1 - \E_\P[X_1]| > k_n^{1/3}\}} \right] \right) \le 2^{-n}.
\]
This yields a game-theoretic composite SLLN for unbounded observations without requiring any moment beyond the first. It may be contrasted with other game-theoretic SLLNs, for example by~\cite{shafer2001probability} and \cite{kumon2008simple}, which are pointwise claims and hold for bounded observations.
\end{remark}

\begin{remark}
    Chung \cite{Chung:SLLN} proved a uniform strong law of large numbers that can be used instead of our line-crossing inequality to prove Theorem~\ref{T_LLN}. Here, we summarize the statement from Shorack's textbook \cite[Theorem 12.1]{Shorack}, which also has a self-contained proof: suppose $(X_t)_{t \in \N}$ is i.i.d.~for each $\P \in \Pcal$, where $\Pcal$ satisfies the uniform integrability condition
    \begin{equation}\label{eq:chung-ui}
        \lim_{K \to \infty}\sup_{\P \in \Pcal} \E_\P[|X_1| \geq K] = 0.
    \end{equation}
    Then, for any $\varepsilon > 0$,
    \begin{equation}\label{eq:chung-slln}
        \lim_{n \to \infty}\sup_{\P \in \Pcal} \P\left(\sup_{m \geq n} |\overline X_m - \E_\P[X_1]| \geq \varepsilon\right) = 0.
    \end{equation}
    Interestingly, our slightly weaker assumption~\eqref{T_LLN_eq1} is already sufficient to yield Chung's conclusion~\eqref{eq:chung-slln}: given a family $\Pcal$ satisfying \eqref{T_LLN_eq1}, define $\Pcal'=\{\P - \E_\P[X_1]: \P\in \Pcal\}$, where we write ``$\P-\E_\P[X_1]$'' as shorthand for the law under $\P$ of the centered sequence $(X_t-\E_\P[X_1])_{t\in\N}$.  
    Then $\Pcal'$ satisfies \eqref{eq:chung-ui}, and applying Chung's result to $\Pcal'$ yields~\eqref{eq:chung-slln}, first with $\Pcal'$ in place of $\Pcal$, but then also for $\Pcal$ itself. Once this has been done, a proof of Theorem~\ref{T_LLN} can be obtained as follows. Up until \eqref{eq_T_LNN_5} the proof is unchanged, but then, instead of \eqref{eq_T_LNN_5}, one uses the triangle inequality to get
    \[
    |\overline X_{k+t} - \overline X_k| \le |\overline X_{k+t} - c| + |\overline X_k - c|
    \]
    for any constant $c$. For each $\P \in \Pcal$ one then takes $c = \E_\P[X_1]$ to get
    \[
    \P(\tau_{k,n} < \infty) \le 2\, \P\left( \sup_{t \in \N} |\overline X_{k+t} - \E_\P[X_1]| \ge \frac{1}{2n} \right).
    \]
    Applying \eqref{eq:chung-slln} now yields $\lim_{k \to \infty} \sup_{\P \in \Pcal} \P(\tau_{k,n} < \infty) = 0$, as required.
\end{remark}

\subsection{The 
insufficiency of test (super)martingales for the
composite strong law}\label{sec:necessity}

Consider any $\Pcal$ that satisfies the conditions of Theorem~\ref{T_LLN}: it contains distributions having a finite mean and satisfying the uniform integrability condition~\eqref{T_LLN_eq1}. The composite strong law states that $\outerLocal(A_{\textnormal{div}})=0$, meaning that the ``erratic'' sequences --- those on which the empirical averages do not converge to a limit --- form a $\outerLocal$-nullset. Theorem~\ref{T_e_process} in turn constructs an explicit e-process $(E_t)_{t \in \N_0}$ --- a nonnegative process such that $\E_\P[ E_\tau ] \leq 1$ for any $\P \in \Pcal$ and stopping time $\tau$ --- that grows to infinity on any erratic sequence. Further, $(\mathbf{1}_{\{E_t > 1/\alpha\}})_{t \in \N_0}$ yields a level-$\alpha$ sequential test for the null 
 hypothesis that the empirical average of the observations converges to some limit.

One important question is whether an e-process was necessary for this purpose, or if a nonnegative martingale (with respect to $\Pcal$) beginning at one would have sufficed. In Ville's original result, $\Pcal$ was a singleton, and hence such martingales do exist and were explicitly constructed. However, in our setting, the question is whether a composite $\Pcal$-martingale exists, which would satisfy the martingale property for \emph{every} $\P \in \Pcal$. 

The answer is an emphatic no: for any sufficiently rich $\Pcal$, there is no nonnegative $\Pcal$-martingale that can grow to infinity on $\outerLocal$-nullsets. In fact, it is possible that the only nonnegative $\Pcal$-martingales are constants (and if we fix their initial value to one, then it must simply remain at one). This fact is not obvious. Ramdas et al.~\cite{ramdas2021testing} prove that for $\Pcal^{01}=\{\text{Ber}(p)^\infty : p \in [0,1]\}$ on the space $\Omega = \{0,1\}^\N$, meaning the set of all possible i.i.d.\ Bernoulli sequences, the only nonnegative $\Pcal^{01}$-martingales are constants. 

If (for example) $\Omega=\mathbb{R}^\mathbb{N}$, then as long as  $\Pcal^{01} \subset \Pcal$, there can still be no nonnegative $\Pcal$-martingale that can grow to infinity on $A_{\text{div}}$. Indeed, $A_{\text{div}}$ contains many binary sequences; on these, no $\Pcal$-martingale can grow, since no $\Pcal^{01}$-martingale can grow.

If the constraint of being a martingale for every $\P \in \Pcal$ is too hard to meet, one could ask if it can be relaxed to a nonnegative $\Pcal$-supermartingale. The answer again is no. The aforementioned paper also proves that the only nonnegative $\Pcal^{01}$-supermartingales are decreasing sequences, and thus the same conclusion holds in our setting. 

However, in that paper, the authors show that it is possible to construct a nontrivial e-process that tests $\Pcal^{01}$ (equivalently, since the e-process property is unaffected by taking convex combinations of distributions, it tests whether the distribution of the data is exchangeable), and that such an e-process is not condemned to stay small: it increases to infinity under a variety of alternatives, such as Markovian ones.

Unlike nonnegative $\Pcal$-(super)martingales, the $\Pcal$-e-process constructed by Theorem~\ref{T_e_process} is also not condemned to stay small, it increases to infinity on any erratic sequence. The above paragraphs demonstrate that e-processes are not only sufficient, but also (in some sense) necessary, in the sense that existing standard martingale concepts are not satisfactory in such richly composite settings.

\section{Summary}

This paper derived a composite generalization of Ville's fundamental martingale theorem that relates measure-theoretic and game-theoretic probability for single distributions. This required the development of an inverse price outer measure of any event, which is the appropriate generalization of the probability of that event in the singleton case. It also required the concept of an e-process, which generalizes an optional stopping property of martingales and supermartingales, and appropriately extends those fundamental concepts to the composite setting. Then, we proved that for any event of inverse price measure zero, there exists an e-process which increases to infinity if that event occurs. Said differently, our results give a rigorous way to find a fair price for indicators of events in a composite setting. 

There exist composite sets $\Pcal$ for which a nonnegative (super)martingale could have sufficed, and e-processes are not required. Examining interesting special classes $\Pcal$, and understanding when supermartingales do or do not suffice, has been the subject of much intense research recently \cite{howard_exponential_2018,howard_uniform_2019,grunwald_safe_2019,ramdas2020admissible,ramdas2021testing,koolen2022log}. In this paper, we show that the composite strong law of large numbers provided a setting where e-processes were necessary and sufficient. 
Future work could extend game-theoretic constructions for the law of the iterated logarithm \cite{miyabe2013law,sasai2019erdHos} to composite settings.

We end with a quote by Ville himself in an interview from the 1980s \cite{crepel2009jean}, reflecting on his seminal work:
\begin{quote}
    My goal, given the $X_n$ and a set to which
the sequence of $X_n$ belongs with probability zero, was to define \emph{(martingales)} $s_n$ so that
they tend to infinity on that set. I insist on the point because it took me so
long to make this way of proceeding understood.
\end{quote}
Our work extends Ville's research program to the realm of composite, nonparametric sets of distributions, and expands the deep connections between measure-theoretic and game-theoretic probability, as developed by~\cite{shafer2001probability,shafer2019game}.


\bibliographystyle{amsplain}
\bibliography{bibs}     




\end{document}